\documentclass[a4paper]{amsart}
\usepackage[active]{srcltx}
\usepackage[all]{xy}
\usepackage{subfig}
\usepackage{hyperref}
\usepackage{mathrsfs}

\usepackage{amsmath}
\usepackage{amssymb,latexsym}
\usepackage{mathrsfs}
\usepackage{graphics}
\usepackage{latexsym}
\usepackage{psfrag}
\usepackage{import}
\usepackage{verbatim}
\usepackage{graphicx}
\usepackage[usenames]{color}
\usepackage{pifont,marvosym}

\theoremstyle{plain}
\newtheorem{lemma}{Lemma}[section]
\newtheorem{theorem}[lemma]{Theorem}
\newtheorem{proposition}[lemma]{Proposition}
\newtheorem{corollary}[lemma]{Corollary}

\theoremstyle{definition}

\newtheorem{definition}[lemma]{Definition}
\newtheorem{remark}[lemma]{Remark}

\numberwithin{equation}{section}

\newcommand{\R}{\mathbb{R}}
\newcommand{\N}{\mathbb{N}}

\newcommand{\supp}{\text{\rm supp}}

\newcommand{\gr}{\textrm{graph}}

\newcommand{\ve}{\varepsilon}

\newcommand{\erre}{\mathbb{R}}
\newcommand{\enne}{\mathbb{N}}

\newcommand{\RCD}{\mathsf{RCD}^{*}}
\newcommand{\f}{\varphi}

\newcommand{\G}{\mathcal{G}}

\newcommand{\Opt}{\mathrm{OptGeo}}
\newcommand{\Geo}{\mathrm{Geo}}
\renewcommand{\r}{\varrho}

\begin{document}
\title[Monge problem in $\RCD$-spaces]{Monge problem in metric measure spaces \\ with Riemannian curvature-dimension condition}

\author{Fabio Cavalletti}

\address{RWTH, Department of Mathematics, Templergraben  64, D-52062 Aachen (Germany)}
\email{cavalletti@instmath.rwth-aachen.de}

\begin{abstract}
We prove the existence of solutions for the Monge minimization problem, addressed in a metric measure space $(X,d,m)$
enjoying the Riemannian curvature-dimension condition $\RCD(K,N)$, with $N < \infty$. 
For the first marginal measure, we assume that $\mu_{0} \ll m$. As a corollary, 
we obtain that the Monge problem and its relaxed version, the Monge-Kantorovich problem, attain the same minimal value.

Moreover we prove a structure theorem for $d$-cyclically monotone sets: neglecting a 
set of zero $m$-measure they do not contain any branching structures, that is, they can be written as the disjoint union 
of the image of a disjoint family  of geodesics.
\end{abstract}

\maketitle

\tableofcontents

\bibliographystyle{plain}

\section{Introduction}

Let $(X,d,m)$ be a metric measure space verifying the Riemannian curvature dimension condition $\mathsf{RCD}^{*}(K,N)$
for $K, N \in \R$ with $N \geq 1$.
%%%%%%%%
In this note we prove the existence of a solution for the following Monge problem: 
given $\mu_{0}, \mu_{1} \in \mathcal{P}(X)$ solve the following minimization problem
\begin{equation}\label{E:Monge1}
\inf_{T_{\sharp}\mu_{0} = \mu_{1}}  \int_{X} d(x,T(x)) \mu_{0}(dx),
\end{equation}
provided $\mu_{0} \ll m$.
More in detail, the minimization of the functional runs over the set of $\mu_{0}$-measurable maps $T: X \to X$ such that 
$T_{\sharp}\mu_{0} = \mu_{1}$, that is 
\[
\mu_{0} ( T^{-1}(A) ) = \mu_{1}(A), \qquad \forall A \in\mathcal{B}(X),
\]
where $\mathcal{B}(X)$ denotes the $\sigma$-algebra of all Borel subsets of $X$. 

On the way to the proof of the existence of an optimal map, we will also 
prove a structure theorem for branching structures inside $d$-cyclically monotone sets.
%%%%%%%%%%%%%%%%%%%%%%%%%%%
Before giving the statements of the two main results of this note and an account on the strategies to prove them, 
we recall some of the (extensive) literature on the Monge minimization problem.

The first formulation for \eqref{E:Monge1} (Monge in 1781) was addressed in $\erre^{n}$ 
with the cost given by the Euclidean norm and the measures $\mu_{0}, \mu_{1}\ll \mathcal{L}^{n}$ were supposed to be supported on two disjoint compact sets. The original problem remained unsolved for a long time. 
In 1978 Sudakov in \cite{sudak} proposed a solution for any distance cost induced 
by a norm, but an argument about disintegration of measures contained in his proof was not correct, see \cite{larm} for details.
Then the Euclidean case was correctly solved by Evans and Gangbo in \cite{evagangbo}, 
under the assumptions that $\textrm{spt}\,\mu_{0} \cap  \textrm{spt}\,\mu_{1} = \emptyset$,  $\mu_{0},\mu_{1} \ll \mathcal{L}^{n}$ and their densities are Lipschitz functions with compact support.
After that, many results reduced the assumptions on the supports of $\mu_{0},\mu_{1}$, see \cite{caffafeldmc} and \cite{trudiwang}. 
The result on manifolds with geodesic cost is obtained in \cite{feldcann:mani}.
The case of a general norm as cost function on $\erre^{n}$ has been solved first in the particular case of crystalline norms in \cite{ambprat:crist}, 
and then in fully generality independently by L. Caravenna in \cite{caravenna:Monge} and by T. Champion and L. De Pascale in 
\cite{champdepasc:Monge}. 

The study of the geodesic metric space framework started with \cite{biacava:streconv}, where the metric space was assumed to be also 
non-branching. There the existence of solutions to \eqref{E:Monge1} 
was obtained for metric spaces verifying the measure-contraction property $\mathsf{MCP}(K,N)$ (for instance the Heisenberg group). 
An application of the results of \cite{biacava:streconv} to the Wiener space can be found in \cite{cava:wiener}.
Then in \cite{cava:nonconv} the problem was studied removing the 
non-branching assumption but obtaining existence of solutions only in a particular case. 

Non-branching metric measure spaces 
enjoying $\mathsf{CD}^{*}(K,N)$ also verify $\mathsf{MCP}(K,N)$. Then from \cite{biacava:streconv} the Monge problem is solved also in that case.
So with respect to the most general known case, we impose a stronger curvature information (namely $\mathsf{RCD}^{*}(K,N)$) 
and we remove the non-branching assumption.

%%%%%%%%%%%%%%%%%%%%%%%%%%%%%%%%%%%%
\subsection{The results}
The nowadays classical strategy to show existence of optimal maps is to relax the integral functional to the larger 
class of transport plans
\[
\Pi(\mu_{0},\mu_{1}): = \{ \pi \in \mathcal{P}(X\times X): (P_{1})_{\sharp}\pi =\mu_{0},  (P_{2})_{\sharp}\pi =\mu_{1}\},
\]
over where the functional we want to minimize has now the following expression
\[
\int d(x,y) \eta(dxdy).
\]
For $i=1,2$, $P_{i}: X \times X \to X$ denotes the projection map on the $i$-th component.
Then assuming that the functional is finite at least on one element of $\Pi(\mu_{0},\mu_{1})$, by linearity in $\eta$ and tightness of $\Pi(\mu_{0},\mu_{1})$, it follows the existence of $\eta_{opt} \in \Pi(\mu_{0},\mu_{1})$
so that 
\[
\int d(x,y) \eta_{opt}(dxdy) = \inf_{\eta \in \Pi(\mu_{0},\mu_{1})}  \int_{X} d(x,y) \eta(dx).
\]
Then the central question, whose positive answer would prove existence of a solution to Monge problem, is whether 
or not $\eta_{opt}$ is supported on the graph of a $m$-measurable map $T: X \to X$. 

The only property characterizing $\eta_{opt}$ inside $\Pi(\mu_{0},\mu_{1})$ is to be concentrated on a $d$-cyclically monotone set.
Hence to build an optimal map we have to start from that.
But while  the Riemannian curvature-dimension condition $\RCD(K,N)$ 
gives crucial information on $d^{2}$-cyclically monotone sets (neglecting a set of measure zero, they are the graph of a measurable map, 
see Section \ref{S:rcd} and references therein), 
nothing is known under this curvature assumption on the structure of $d$-cyclically monotone sets.
In particular what we would like to exclude is the presence of branching structures. 
Note that the first result proving absence of branching geodesics assuming 
a curvature condition, in that case strong $\mathsf{CD}(K,\infty)$, is contained in \cite{rajasturm:branch}.

%in \cite{giglirajasturm:optimalmaps} it is proved that $\mathsf{RCD}^{*}(K,N)$ implies 
%the measure-contraction property 
%and, as already pointed out, in \cite{biacava:streconv} the existence of an optimal map for the Monge problem has been shown in the framework of 
%non-branching metric measure spaces $(X,d,m)$ enjoying the $\mathsf{MCP}(K,N)$. 

The strategy we will follow is: prove that $d$-cyclically monotone sets do not have branching structures $m$-almost everywhere;
then use the approach with Disintegration Theorem (see for instance \cite{biacava:streconv} and references therein) to reduce the Monge problem to a family of 1-dimensional Monge problem. 
There one can apply the 1-dimensional theory. 
Thanks to the curvature assumption we can prove a suitable property for the first marginal measures 
and obtain the existence of the 1-dimensional optimal maps, one for each 1-dimensional Monge problem.  
Then gluing together all the one-dimensional optimal maps, one gets an optimal map $T:X \to X$ solution of the Monge problem \eqref{E:Monge1}.
A more precise program on the use of Disintegration Theorem in the Monge problem will be given in Section \ref{S:dmonotone}.

We conclude this introductory part stating the two main results we will prove.
The first is about the structure of the $d$-cyclically monotone set associated to a Kantorovich potential $\f^{d}$ 
for the problem \eqref{E:Monge1}.

\begin{theorem}\label{T:1}
Let $(X,d,m)$ be a metric measure space verifying $\RCD(K,N)$ for some $K,N \in \R$, with $N\geq 1$.
Let moreover $\Gamma$ be a $d$-cyclically monotone set as \eqref{E:Gamma} and let $\mathcal{T}_{e}$ be the set of all points moved by $\Gamma$ as in Definition \ref{D:transport}.
Then there exists $\mathcal{T} \subset \mathcal{T}_{e}$ that we call the transport set such that 
\[
m(\mathcal{T}_{e} \setminus \mathcal{T}) = 0, \quad 
\]
and for all $x \in \mathcal{T}$, the transport ray $R(x)$ is formed by a single geodesic and for $x\neq y$, both in $\mathcal{T}$, either $R(x) = R(y)$
or $R(x) \cap R(y)$ is contained in the set of initial points $a \cup b$ as Definition \ref{D:transport}.
\end{theorem}

All the terminology used in Theorem \ref{T:1} will be introduced in Section \ref{S:dmonotone}.
Taking advantage of Theorem \ref{T:1} we then obtain the following

\begin{theorem}\label{T:2}
Let $(X,d,m)$ be a metric measure space verifying $\RCD(K,N)$ for $N<\infty$. Let $\mu_{0},\mu_{1} \in \mathcal{P}(X)$ 
with $W_{1}(\mu_{0},\mu_{1}) < \infty$ and $\mu_{0}\ll m$. 
Then there exists a Borel map $T: X \to X$ such that $T_{\sharp} \mu_{0} = \mu_{1}$ and 
\[
\int_{X} d(x,T(x)) \mu_{0}(dx) =  \int_{X \times X} d(x,y) \eta_{opt}(dxdy).
\]
\end{theorem}

In the previous theorem, $W_{1}$ denotes the $L^{1}$-Wasserstein distance on the space of probability measures on $(X,d)$. 

A straightforward corollary of Theorem \ref{T:2} is that the relaxation to the set of transference plan $\Pi(\mu_{0},\mu_{1})$ 
does not lower the value of the minimum:
\begin{align*}
\inf_{T_{\sharp}\mu_{0} = \mu_{1}}  \int_{X} d(x,T(x)) \mu_{0}(dx) \leq &~ \int d(x,T(x)) \mu_{0}(dx)  \crcr
= &~ \min_{\eta \in \Pi(\mu_{0},\mu_{1})}  \int_{X} d(x,y) \eta(dx)  \crcr
\leq  &~ \inf_{T_{\sharp}\mu_{0} = \mu_{1}}  \int_{X} d(x,T(x)) \mu_{0}(dx).
\end{align*}
Hence 
\begin{equation}\label{E:K=M}
\min_{T_{\sharp}\mu_{0} = \mu_{1}}  \int_{X} d(x,T(x)) \mu_{0}(dx) =\min_{\eta \in \Pi(\mu_{0},\mu_{1})}  \int_{X} d(x,y) \eta(dx).
\end{equation}

%%%%%%%%%%%%%%%%%%%%%%%%%%%%%%%
\bigskip
\bigskip

As it will be clear from their proofs, the results contained in Theorem \ref{T:1} and Theorem \ref{T:2} can be obtained 
omitting the $\mathsf{RCD}^{*}$ condition and assuming instead the metric measure space to satisfy the strong $\mathsf{CD}^{*}(K,N)$ condition.
Even if strong $\mathsf{CD}^{*}(K,N)$ is a more general condition than $\mathsf{RCD}^{*}$, the latter is stable with respect to  
measured Gromov-Hausdorff convergence.
Hence we have decided in its favor to state and prove the results contained in this note.

\bigskip
%%%%%%%%%%%%%%%%%%%%%%%%%%%%%%
\emph{The author wish to thank Tapio Rajala for a discussion on an early version of this note.}
%%%%%%%%%
%

\section{$\RCD$ spaces}\label{S:rcd}
Here we briefly give some references for $\RCD(K,N)$ and state some of the main properties of metric measure spaces verifying it. 

%%%%%
Few notations:
we will denote with $\Geo(X) \subset C([0,1], X)$ the space of geodesics endowed with uniform topology and
for a Borel set $F \subset X \times X$, we will often use the notation $F(x)$ for $P_{2}(F \cap \{x \}\times X )$.

For $\mu_{0},\mu_{1} \in \mathcal{P}_{2}(X)$ we consider the following set of optimal geodesics:
\[
\Opt(\mu_{0},\mu_{1}) : = \left\{ \nu \in \mathcal{P}(\Geo(X)) :   \int d(x,y)^{2} (e_{s},e_{t})_{\sharp}\nu   =  (t-s)^{2} W_{2}^{2}(\mu_{0},\mu_{1}), \, 0 \leq s\leq t \leq 1 \right\},
\]
where $W_{2}$ is the $L^{2}$-Wasserstein distance. We write $\nu \in \Opt$,  
if $\nu \in \Opt(e_{0}\,_{\sharp}\nu, e_{1}\,_{\sharp}\nu)$.

Sturm and independently Lott and Villani, using the $L^{2}$-Wasserstein space, introduced a class of metric measure space verifying a generalized
curvature condition called curvature-dimension condition $\mathsf{CD}(K,N)$, with $K,N \in \R$ and $N\geq 2$. The condition models a 
lower bound on Ricci curvature and an upper bound on the dimension. 
See \cite{sturm:MGH1, sturm:MGH2} and \cite{villott:curv} for the precise definitions.
Then a variant called reduced curvature-dimension condition, denoted with $\mathsf{CD}^{*}(K,N)$, has been introduced in \cite{sturm:loc}.

The Riemannian Curvature Dimension condition $\mathsf{RCD}(K,\infty)$, has been introduced by L. Ambrosio, N. Gigli and 
G. Savar\'e in \cite{ambrgisav:rcd}. 
Then the finite dimensional has been studied in \cite{gigli:laplacian, gigli:spliting}
and the precise $\mathsf{RCD}^{*}(K,N)$ has been defined in \cite{Erbarkuwstu:RCD} and \cite{ambmondsav:RCD} with two different approaches.
We refer to these fundamental papers for the precise definitions. 
Here we will make use of some property enjoyed by this class of spaces.

The following theorem is taken from \cite{rajasturm:branch}.
%%%%%%%%%%%%%%%%%%%%%%%%%%%%%%%%%%%%
\begin{theorem}\label{T:existenceinfinity}
Let $(X,d,m)$ be a $\mathsf{RCD}(K,\infty)$ space and $\mu_{0},\mu_{1} \in \mathcal{P}_{2}(X)$ be two measures absolutely continuous w.r.t. $m$. 
Then there exists a unique $\nu \in \Opt(\mu_{0},\mu_{1})$ and this plan is induced by a map and is concentrated on a set of non-branching geodesics.
\end{theorem}

For $\nu \in \Opt(\mu_{0},\mu_{1})$  to be concentrated on a set of non-branching geodesic means that for any $t \in [0,1]$ 
the evaluation map $e_{t}$ restricted to $\supp(\pi)$ is invertible, that is there exists a Borel map 
$(e_{t})^{-1} : \supp(\mu_{t}) \to \Geo(X)$ such that 
\[
\nu = \left( (e_{t})^{-1} \right)_{\sharp} (e_{t})_{\sharp} \nu,
\]
for every $t \in [0,1]$. 
%%%%%%%%%%%

Theorem \ref{T:existenceinfinity} has been used in \cite{giglirajasturm:optimalmaps} to prove the following localization result for entropy inequality of $\RCD(K,N)$ spaces.

\begin{proposition}\label{P:localization}
Let $(X,d,m)$ be an $\RCD(K,N)$ space and $\mu_{i} = \r_{i}m \in \mathcal{P}_{2}(X)$, $i = 0,1$ be two given measures. 
Let $\nu \in \Opt(\mu_{0},\mu_{1})$ be the unique optimal geodesic plan of Theorem \ref{T:existenceinfinity}.
If $\mu_{t} = (e_{t})_{\sharp} \pi$, then $\mu_{t} \ll m$ for all $t \in [0,1]$ and if $\mu_{t} =\r_{t} m$ then 
\[
\r_{r}(\gamma_{r})^{-\frac{1}{N}} \geq \r_{s}(\gamma_{s})^{-\frac{1}{N}}  \sigma_{K,N}^{(\frac{t-r}{t-s})}(d(\gamma_{s},\gamma_{t}))
+\r_{t}(\gamma_{t})^{-\frac{1}{N}}  \sigma_{K,N}^{(\frac{r-s}{t-s})}(d(\gamma_{s},\gamma_{t})), \qquad  \nu - a.e. \gamma,
\]
for all $0\leq s \leq r \leq t \leq 1$.
\end{proposition}

Again in \cite{giglirajasturm:optimalmaps} it is proven that if $N<\infty$ and the first marginal is absolutely continuous with respect to $m$, then there exists a unique optimal plan and it is concentrated on the graph of a Borel function. 
The optimal plan is also induced by an element of $\pi \in \Opt$ concentrated on a set of non-branching geodesics.
Note that all these results are for geodesics in the $L^{2}$-Wasserstein space, 
while the object of our investigation are $d$-cyclically monotone sets, usually having lower ``regularity'' than 
$d^{2}$-cyclically monotone sets.
Finally note that from $\RCD(K,N)$ it follows that $(X,d,m) = (\supp(m),d,m)$, and the metric space $(X,d)$ 
is geodesic and proper (provided $N<\infty$).

From now on we will assume $(X,d,m)$ to verify $\RCD(K,N)$ for some $K,N \in \R$ with $N \geq 2$.

%%%%%%%%%%%%%%%%%%%%%%%%%%%%%%
\section{$d$-geodesics and $d^{2}$-geodesics}\label{S:dmonotone}
To avoid the trivial case we can assume that the two marginal measures have finite $L^{1}$-Wasserstein distance, 
$W_{1}(\mu_{0},\mu_{1}) < \infty$. Consequently we infere
the existence of  $\eta \in \Pi(\mu_{0},\mu_{1})$,
such that
\[
\int_{X\times X} d(x,y)\eta(dxdy) = \inf \left\{ \int_{X\times X} d(x,y)\pi(dxdy) : \pi \in \Pi(\mu_{0},\mu_{1}) \right\} 
= W_{1}(\mu_{0},\mu_{1}),
\]
where $\Pi(\mu_{0},\mu_{1})$ is the set of transport plans,
\[
\Pi(\mu_{0},\mu_{1}): = \{ \pi \in \mathcal{P}(X\times X): (P_{1})_{\sharp}\pi =\mu_{0},  (P_{2})_{\sharp}\pi =\mu_{1}\}.
\]

The set of optimal transport plans, i.e. realizing the previous identity, will be denoted with $\Pi_{opt}(\mu_{0},\mu_{1})$.
Since the cost is finite, we can also assume the existence of a Kantorovich potential, that is a $1$-Lipschitz function 
$\f^{d} : X \to \erre$, such that 
\[
\eta \in \Pi_{opt}(\mu_{0},\mu_{1}) \iff \eta \left(  \{ (x,y) \in X\times X : \f^{d}(x) - \f^{d}(y) = d(x,y) \} \right) = 1.
\]
We also use the following notation: 
\begin{equation}\label{E:Gamma}
\Gamma : = \{ (x,y) \in X\times X : \f^{d}(x) - \f^{d}(y) = d(x,y) \}.
\end{equation}
Almost by definition, the set $\Gamma$ is a $d$-cyclically monotone set.

The following is a standard fact of $d$-cyclically monotone sets.
\begin{lemma}\label{L:cicli}
Let $(x,y) \in X\times X$ be an element of $\Gamma$. Let $\gamma \in \Geo(X)$ be such that $\gamma_{0} = x$ 
and $\gamma_{1}=y$. Then
\[
(\gamma_{s},\gamma_{t}) \in \Gamma, 
\]
for all $0\leq s \leq t \leq 1$.
\end{lemma}
\begin{proof}
Take $0\leq s \leq t \leq 1$ and note that
\begin{align*}
 \f^{d}(\gamma_{s})& - \f^{d}(\gamma_{t})\crcr
 = &~ \f^{d}(\gamma_{s}) - \f^{d}(\gamma_{t}) + \f^{d}(\gamma_{0}) - \f^{d}(\gamma_{0})
+ \f^{d}(\gamma_{1}) - \f^{d}(\gamma_{1})\crcr
\geq &~d(\gamma_{0},\gamma_{1}) - d(\gamma_{0},\gamma_{s}) - d(\gamma_{t},\gamma_{1}) \crcr
= &~ d(\gamma_{s},\gamma_{t}).
\end{align*}
The claim follows.
\end{proof}

It is therefore natural to consider the set of geodesics $G \subset \Geo(X)$ such that 
\[
\gamma \in G \iff \{ (\gamma_{s},\gamma_{t}) : 0\leq s \leq t \leq 1  \} \subset \Gamma,
\]
that is $G : = \{ \gamma \in \Geo(X) : (\gamma_{0},\gamma_{1}) \in \Gamma \}$.

We now recall some definitions, already given in \cite{biacava:streconv}, 
that will be needed to describe the structure of $\Gamma$.

\begin{definition}\label{D:transport}
We define the set of \emph{transport rays} by 
\[
R = \Gamma \cup \Gamma^{-1},
\]
where $\Gamma^{-1}= \{ (x,y) \in X \times X : (y,x) \in \Gamma \}$. The set of \emph{initial points} and \emph{final points} 
respectively by  
\begin{align*}
a :=& \{ z \in X: \nexists \, x \in X, (x,z) \in \Gamma, d(x,z) > 0  \}, \crcr
b :=& \{ z \in X: \nexists \, x \in X, (z,x) \in \Gamma, d(x,z) > 0 \}.
\end{align*}
The set of \emph{end points} is $a \cup b$.
We also define 
%the \emph{transport set}
%\[
%\mathcal{T} = P_{1}(\Gamma \setminus \{ (x,y) \in X^{2} : d(x,y) = 0 \}) \cap 
%P_{1}(\Gamma^{-1}\setminus \{ (x,y) \in X^{2} : d(x,y) = 0 \})
%\]
%and 
the \emph{transport set with end points}: 
\[
\mathcal{T}_{e} = P_{1}(\Gamma \setminus \{ x = y \}) \cup 
P_{1}(\Gamma^{-1}\setminus \{ x=y \}).
\]
where $\{ x = y\}$ stays for $\{ (x,y) \in X^{2} : d(x,y) = 0 \}$.
\end{definition}

\begin{remark}\label{R:regularity}
Here we discuss the measurability of the sets introduced in Definition \ref{D:transport}.
Since $\f^{d}$ is $1$-Lipschitz, $\Gamma$ is closed and therefore $\Gamma^{-1}$ and $R$ are closed as well.
Moreover thanks to curvature assumption the space is proper, hence the sets $\Gamma, \Gamma^{-1}, R$ are $\sigma$-compact.

Then we look at the set of initial and final points: 
\[
a = P_{2} \left( \Gamma \cap \{ (x,z) \in X\times X : d(x,z) > 0 \} \right)^{c}, \qquad 
b = P_{1} \left( \Gamma \cap \{ (x,z) \in X\times X : d(x,z) > 0 \} \right)^{c}.
\]
Since $\{ (x,z) \in X\times X : d(x,z) > 0 \} = \cup_{n} \{ (x,z) \in X\times X : d(x,z) \geq 1/n \}$, it follows that both $a$ and $b$ are the complement of  $\sigma$-compact sets. Hence $a$ and $b$ are Borel sets. Reasoning as before, it follows that 
%both $\mathcal{T}$ and 
$\mathcal{T}_{e}$ is a $\sigma$-compact set. 
%To conclude we observe that
%\[
%\mathcal{T} = \mathcal{T}_{e} \setminus  \left( a \cup b \right), \qquad 
%\mathcal{T}_{e} = \mathcal{T}  \cup a \cup b.
%\]
\end{remark}

Next Lemma permits to reduce the analysis of the existence of solutions of the Monge problem on the whole $X$ 
to the same problem restricted to the transport set with end points.

\begin{lemma} \label{L:mapoutside}
Let $\eta \in \Pi_{opt}(\mu_{0},\mu_{1})$, then
\[
\eta(\mathcal{T}_e \times \mathcal{T}_e \cup \{x = y\}) = 1.
\]
\end{lemma}

\begin{proof}
It is enough to observe that if $(z,w) \in \Gamma$ with $z \neq w$, then $w \in \Gamma(z)$ and $z \in \Gamma^{-1}(w)$ 
and therefore
\[
(z,w) \in \mathcal{T}_{e}\times \mathcal{T}_{e}.
\]
Hence $\Gamma \setminus \{x = y\}  \subset   \mathcal{T}_e \times \mathcal{T}_e$. Since $\eta(\Gamma) =1$, 
the claim follows.
\end{proof}

As a consequence, $\mu_{0}(\mathcal{T}_e) = \mu_{1}(\mathcal{T}_e)$ and any optimal map $T$ such that 
$T_\sharp \mu_{0} \llcorner_{\mathcal{T}_e}= \mu_{1} \llcorner_{\mathcal{T}_e}$ 
can be extended to an optimal map $T'$ with $ T^{'}_\sharp \mu_{0} = \mu_{1}$ with the same cost by setting
\begin{equation}
\label{E:extere}
T'(x) =
\begin{cases}
T(x) & x \in \mathcal{T}_e \crcr
x & x \notin \mathcal{T}_e.
\end{cases}
\end{equation}

\bigskip
Using the terminology introduced so far,
we explain the strategy we will follow to prove existence of an optimal map: 
first we need to find a suitable subset $\mathcal{T}$ of $\mathcal{T}_{e}$ called the \emph{transport set}, with 
$m(\mathcal{T}_{e} \setminus \mathcal{T}) = 0$, enjoying better geometric properties than $\mathcal{T}_{e}$ (remove branching geodesics). Then
\begin{enumerate}
\item prove that for every $x\in \mathcal{T}$ 
there exists only one unparametrized geodesic passing through $x$ and contained in $\mathcal{T}_{e}$;
\item reduce the $L^{1}$ optimal transport problem to a 1-dimensional $L^{1}$ optimal transport problem along each unparametrized geodesic;
\item prove regularity (i.e. absence of atoms) of conditional probability for the 1-dimensional $L^{1}$ optimal transport problem.
\end{enumerate}
Once these three points have been accomplished, one obtains the existence of an optimal map for each 
1-dimensional $L^{1}$ optimal transport problem, by considering for instance the unique monotone rearrangement between the two 
1-dimensional measures. Then glueing all the 1-dimensional optimal maps, one obtain a global optimal map. 

We recall the one dimensional result for the Monge problem \cite{villa:Oldnew}, that will be used as a building block.

\begin{theorem}
\label{T:oneDmonge}
Let $\mu_{0}$, $\mu_{1}$ be probability measures on $\erre$, $\mu_{0}$ with no atoms, and let
\[
H(s) := \mu((-\infty,s)), \quad F(t) := \nu((-\infty,t)),
\]
be the left-continuous distribution functions of $\mu_{0}$ and $\mu_{1}$ respectively. Then the following holds.
\begin{enumerate}
\item The non decreasing function $T : \erre \to \erre \cup [-\infty,+\infty)$ defined by
\[
T(s) := \sup \big\{ t \in \erre : F(t) \leq H(s) \big\}
\]
maps $\mu_{0}$ to $\mu_{1}$. Moreover any other non decreasing map $T'$ such that $T'_\sharp \mu_{0} = \mu_{1}$ coincides with $T$ on the support of $\mu_{0}$ up to a countable set.
\item If $\phi : [0,+\infty] \to \erre$ is non decreasing and convex, then $T$ is an optimal transport relative to the cost $c(s,t) = \phi(|s-t|)$. Moreover $T$ is the unique optimal transference map if $\phi$ is strictly convex.
\end{enumerate}
\end{theorem}

%The previous program can be rewritten in the following manner. Point $(1)$ is equivalent to say that the set 
%$R$ is an equivalence relation over a well-prepared subset of $\mathcal{T}$ 
%and each equivalence class is an unparametrized geodesic. 
%Then one can disintegrate the reference measure as follows,
%\[
%m\llcorner_{\mathcal{T}} = \int m_{\alpha} q(d\alpha), 
%\]
%with $m_{\alpha}$ supported on a one
%
%\subsection{$L^{2}$-geodesics via $d$-monotonicity}

\section{The transport set}
We now prove that the set of transport rays $R$
is an equivalence relation on a subset of $\mathcal{T}_{e}$. In order to do so, we study the branching geodesics in $\Gamma$. 
The presence of branching structures inside $\Gamma$ can be modeled by the existence of  
$x, z,w \in \mathcal{T}_{e}$ such that 
\[
(x,z), (x,w) \in \Gamma, \quad (z,w) \notin R.
\]
Actually the previous condition only describes branching in the direction given by $\Gamma$. 
Branching in the direction of $\Gamma^{-1}$ will be treated analogously. 

In the next Lemma, using Lemma \ref{L:cicli}, we prove that, once a branching happens, 
there exists two distinct geodesics, both contained in $\Gamma(x)$, 
that are not in relation in the sense of $R$. Recall that $G \subset \Geo(X)$ is the set of geodesics $\gamma$ such that 
\[
(\gamma_{s},\gamma_{t}) \in \Gamma, \qquad 0\leq s\leq t \leq 1.
\]

\begin{lemma}\label{L:geoingamma}
Let $x \in \mathcal{T}_{e}$ and $z,w \in \mathcal{T}_{e}$ be such that $z,w \in \Gamma(x)$ and $(z,w) \notin R$. 
Then there exist two distinct geodesics $\gamma^{1},\gamma^{2} \in \Geo(X)$ 
such that 
\begin{itemize}
\item[-] $\gamma^{1}, \gamma^{2} \in G$; 
\item[-] $(x,\gamma_{s}^{1}), (x,\gamma_{s}^{2}) \in \Gamma$ for all $s \in [0,1]$;
\item[-] $(\gamma_{s}^{1},\gamma^{2}_{s}) \notin R$ for all $s \in [0,1]$; 
\item[-] $\f^{d}(\gamma^{1}_{s}) = \f^{d}(\gamma^{2}_{s})$ for all $s \in [0,1]$.
\end{itemize}
Moreover both geodesics are non-constant.
\end{lemma}

\begin{proof}
Since $z,w \in \Gamma(x)$, from Lemma \ref{L:cicli} there exist two geodesics $\gamma^{1},\gamma^{2} \in G$ such that 
\[
\gamma^{1}_{0} = \gamma^{2}_{0} = x, \quad \gamma^{1}_{1} = z, \quad \gamma^{2}_{1} = w.
\]
Since $(z,w) \notin R$, necessarily both $z$ and $w$ are different from $x$ and $x$ is not a final point, that is $x \notin b$. So the previous geodesics are not constant.
Since $z$ and $w$ can be exchanged, we can also assume that $\f^{d}(z) \geq \f^{d}(w)$.
Since $z \in \Gamma(x)$, $\f^{d}(x) \geq \f^{d}(z)$ and by continuity there exists 
$s_{2} \in (0,1]$ such that 
\[
\f^{d}(z) = \f^{d}(\gamma^{2}_{s_{2}}).
\]
Note that $z \neq \gamma^{2}_{s_{2}}$, otherwise $w \in \Gamma(z)$ and therefore $(z,w) \in \R$.
Moreover still $(z,\gamma^{2}_{s_{2}}) \notin R$. Indeed if the contrary was true, then 
\[
0= |\f^{d}(z) - \f^{d}(\gamma^{2}_{s_{2}}) | = d(z,\gamma^{2}_{s_{2}}),
\]
that is a contradiction with $z \neq \gamma^{2}_{s_{2}}$.

So by continuity there exists $\delta > 0$ such that 
\[
\f^{d} (\gamma^{1}_{1-s} ) = \f^{d} (\gamma^{2}_{s_{2}(1-s)} ), \qquad d(\gamma^{1}_{1-s}, \gamma^{2}_{s_{2}-s}) > 0,
\]
for all $0 \leq s \leq \delta$. 

Hence reapplying the previous argument $(\gamma^{1}_{1-s}, \gamma^{2}_{s_{2}(1-s)}) \notin R$.
The curve $\gamma^{1}$ and $\gamma^{2}$ of the claim are then obtained properly restricting and rescaling the geodesic $\gamma^{1}$ and $\gamma^{2}$ considered so far.
\end{proof}

There is a measurable correspondence between points of branching and couples of geodesics.
To prove it we need the following selection result, Theorem 5.5.2 of \cite{Sri:courseborel}, page 198.

\begin{theorem}
\label{T:vanneuma}
Let $X$ and $Y$ be Polish spaces, $F \subset X \times Y$ analytic, and $\mathcal{A}$ the $\sigma$-algebra generated by the analytic subsets of X. Then there is an $\mathcal{A}$-measurable section $u : P_1(F) \to Y$ of $F$.
\end{theorem}
Recall that given $F \subset X \times Y$, a \emph{section $u$ of $F$} is a function from $P_1(F)$ to $Y$ such that $\textrm{graph}(u) \subset F$. Here $\mathcal{A}$ denotes the $\sigma$-algebra generated by the analytic subsets of $(X,d)$.

\begin{lemma}\label{L:selectiongeo}
Consider the set of possible branching points defined as follows
\[
A_{+} : = \{ x \in \mathcal{T}_{e} : \exists z,w \in \Gamma(x), (z,w) \notin R \}.
\]
Then there exists 
an $m$-measurable map $u : A_{+} \mapsto G \times G$ such 
that if $u(x) = (\gamma^{1},\gamma^{2})$ then
\begin{itemize}
\item[-] $(x,\gamma_{s}^{1}), (x,\gamma_{s}^{2}) \in \Gamma$ for all $s \in [0,1]$;
\item[-] $(\gamma_{s}^{1},\gamma^{2}_{s}) \notin R$ for all $s \in [0,1]$; 
\item[-] $\f^{d}(\gamma^{1}_{s}) = \f^{d}(\gamma^{2}_{s})$ for all $s \in [0,1]$.
\end{itemize}
Moreover both geodesics are non-constant.
\end{lemma}

\begin{proof}
Since $G = \{ \gamma \in \Geo(X) : (\gamma_{0},\gamma_{1}) \in \Gamma \}$,
and that $\Gamma \subset X \times X$ is closed, the set $G$ is a complete and separable metric space.
Consider now the set 
\begin{align*}
F:= &~ \{ (x,\gamma^{1},\gamma^{2}) \in \mathcal{T}_{e}\times G \times G : (x,\gamma^{1}_{0}), (x,\gamma^{2}_{0}) \in \Gamma \} \crcr
&~ \cap\left( X\times \{ (\gamma^{1},\gamma^{2}) \in G\times G : d(\gamma^{1}_{1},\gamma^{2}_{1})>0 \} \right) \crcr
&~ \cap\left( X\times \{ (\gamma^{1},\gamma^{2}) \in G\times G : d(\gamma^{1}_{0},\gamma^{2}_{0})>0 \} \right) \crcr
&~ \cap\left( X\times \{ (\gamma^{1},\gamma^{2}) \in G\times G : \f^{d}(\gamma^{1}_{i}) = \f^{d}(\gamma^{2}_{i}), \, i =0,1 \} \right).
\end{align*}
It follows from Remark \ref{R:regularity} that $F$ is $\sigma$-compact and from Lemma \ref{L:geoingamma}, 
\[
F \cap  \left( \{x\} \times G\times G \right) \neq \emptyset
\]
for all $x \in A_{+}$. Theorem \ref{T:vanneuma} infer the existence of an $\mathcal{A}$-measurable selection $u$ of $F$.
Then since $A_{+} = P_{1}(F)$ and in particular $S$ is $m$-measurable, the claim follows.
\end{proof}

Note that in the proof of Lemma \ref{L:selectiongeo} we have also shown that $A_{+}$ is a $\sigma$-compact set.

We recall here the crucial construction, already introduced in \cite{cav:decomposition}, that permits to apply the known results on the structure 
of $d^{2}$-cyclically monotone sets to $d$-cyclically monotone one.

\begin{lemma}\label{L:12monotone}
Let $\Delta \subset \Gamma$ be any set so that: 
\[
(x_{0},y_{0}), (x_{1},y_{1}) \in \Delta \quad \Rightarrow \quad (\f^{d}(y_{1}) - \f^{d}(y_{0}) )\cdot (\f^{d}(x_{1}) - \f^{d}(x_{0}) ) \geq 0.
\]
Then $\Delta$ is $d^{2}$-cyclically monotone.
\end{lemma}

\begin{proof} It follows directly from the hypothesis of the Lemma that the set
\[
\{ (\f^{d}(x), \f^{d}(y) ) :   (x,y) \in \Delta \} \subset \erre \times \erre
\]
is $|\cdot|^{2}$-cyclically monotone, where $|\cdot|$ denotes the modulus. 
Then for $\{(x_{i},y_{i})\}_{ i \leq N} \subset \Delta$, since $\Delta \subset \Gamma$,
it holds
\begin{align*} 
\sum_{i=1}^{N} d^{2}(x_{i},y_{i}) = &~ \sum_{i =1}^{N}|\f^{d}(x_{i}) - \f^{d}(y_{i})|^{2} \crcr
\leq&~ \sum_{i =1}^{N}|\f^{d}(x_{i}) - \f^{d}(y_{i+1})|^{2} \crcr
\leq &~ \sum_{i=1}^{N} d^{2}(x_{i},y_{i+1}),
\end{align*}
where the last inequality is given by the 1-Lipschitz regularity of $\f^{d}$. The claim follows.
\end{proof}

The first consequence of Lemma \ref{L:12monotone} is the following

\begin{proposition}\label{P:nobranch}
The set 
\[
A_{+} := \{ x \in \mathcal{T}_{e} : \exists z,w \in \Gamma(x), (z,w) \notin R \}
\]
has $m$-measure zero.
\end{proposition}

\begin{proof}
{\it Step 1.} Suppose by contradiction that $m(A_{+})>0$. 
By definition of $A_{+}$, thanks to Lemma \ref{L:geoingamma} and Lemma \ref{L:selectiongeo}, 
for every $x \in A_{+}$ there exist two non-constant geodesics $\gamma^{1},\gamma^{2} \in \G(X)$ 
such that 
\begin{itemize}
\item[-] $\gamma^{1}, \gamma^{2} \in G$; 
\item[-] $(x,\gamma_{s}^{1}), (x,\gamma_{s}^{2}) \in \Gamma$ for all $s \in [0,1]$;
\item[-] $(\gamma_{s}^{1},\gamma^{2}_{s}) \notin R$ for all $s \in [0,1]$; 
\item[-] $\f^{d}(\gamma^{1}_{s}) = \f^{d}(\gamma^{2}_{s})$ for all $s \in [0,1]$.
\end{itemize}
Moreover the map $A_{+} \ni x \mapsto u(x) : = (\gamma^{1},\gamma^{2}) \in G^{2}$ is $m$-measurable.

Then again by inner regularity of compact sets, we can assume that the previous map is continuous and in particular the functions
\[
A_{+} \ni x \mapsto \f^{d}(\gamma^{i}_{j}) \in \R, \qquad  i =1,2, \ j = 0,1
\]
are all continuous. Put $\alpha_{x} = \f^{d}(\gamma^{1}_{0})$ and $\beta_{x} = \f^{d}(\gamma^{1}_{1})$ and note that 
$\alpha_{x} > \beta_{x}$.
Now we want to show the existence of a subset $B \subset A_{+}$, still with $m(B) > 0$, such that 
\[
\sup_{x \in B} \beta_{x} < \inf_{x\in B} \alpha_{x}.
\]
By continuity of $\alpha$ and $\beta$, a set $B$ verifying the previous inequality
can be obtained considering the set $A_{+} \cap B_{r}(x)$, for $x \in A_{+}$ and for $r$ sufficiently small.
Since $m(A_{+})>0$, for $m$-a.e. $x \in A_{+}$ the set $A_{+}\cap B_{r}(x)$ has positive $m$-measure.
So the existence of $B \subset A_{+}$ enjoying the aforementioned properties follows.

{\it Step 2.} Let $I = [c,d]$ be a non trivial interval such that 
\[
\sup_{x \in B} \beta_{x} < c < d <\inf_{x\in B} \alpha_{x}. 
\]
Then by construction for all $x \in B$ the image of the composition of the geodesics $\gamma^{1}$ and $\gamma^{2}$ with $\f^{d}$
contains the interval $I$: 
\[
I \subset \{ \f^{d}(\gamma^{i}_{s}) : s \in [0,1] \}, \qquad i = 1,2.
\] 
Now let $T : \R \to \R$ be a monotone map such that $T (\f^{d}(B)) = I$. 
Then we can consider the following function: to each $x \in B$ we associate $s(x) \in [0,1]$ such that 
\[
\f^{d} (\gamma^{1}_{s(x)}) = T(\f^{d}(x)).
\]
Note that the map $s \mapsto \f^{d} (\gamma^{1}_{s})$ is strictly decreasing and 
\[
T(\f^{d}(x)) \in I \subset \{ \f^{d}(\gamma^{1}_{s}) : s \in [0,1] \},
\]
therefore the function $s(x)$ is well defined and by construction $\f^{d}(\gamma^{2}_{s(x)}) =  \f^{d}(\gamma^{1}_{s(x)}) = T(\f^{d}(x))$.

We can now define on $B$ two transport maps $T^{1}$ and $T^{2}$ by
\[
B \ni x \mapsto T^{i}(x) : = \gamma^{i}_{s(x)}, \qquad i =1,2.
\]
Accordingly we define the transport plan
\[
\eta : = \frac{1}{2} \left(   (Id, T^{1})_{\sharp} m_{B} + (Id, T^{2})_{\sharp} m_{B}   \right),
\]
where $m_{B} : = m(B)^{-1} m \llcorner_{B}$.

{\it Step 3.} The support of $\eta$ is $d^{2}$-cyclically monotone. To prove it we will use Lemma \ref{L:12monotone}. 
The measure $\eta$ is concentrated on the set
\[
\Delta : = \{ (x,\gamma^{1}_{s(x)}) : x \in B \} \cup  \{ (x,\gamma^{2}_{s(x)}) : x \in B \} \subset \Gamma.
\]
Possibly restricting again the set $B$, we can assume $T^{1}$ and $T^{2}$ to be continuous and therefore 
$\Delta$ to be the support of $\eta$.
Then take any two couples $(x_{0},y_{0}), (x_{1},y_{1}) \in \Delta$ and by definition of $T$:
\[
\f^{d}(y_{1}) - \f^{d}(y_{0}) =T(\f^{d}(x_{1})) - T(\f^{d}(x_{0})).
\]
Since $T$ is monotone it follows that $\left( T(\f^{d}(x_{1})) - T(\f^{d}(x_{0})) \right) \left( \f^{d}(x_{1}) - \f^{d}(x_{0}) \right) \geq 0$
and Lemma \ref{L:12monotone} can be applied to $\Delta$. 
Hence $\Delta$ is $d^{2}$-monotone. Hence $\eta$ is optimal with $(P_{1})_{\sharp}\eta \ll m$ and this is a contradiction with 
the curvature property $\RCD(K,N)$ that implies that every optimal transportation is induced by a map.
The claim follows.
\end{proof}

Thanks to the symmetry of the statement of Proposition \ref{P:nobranch}, it can be proven that also the set 
\[
A_{-}: = \{ x \in \mathcal{T}_{e} : \exists z,w \in \Gamma(x)^{-1}, (z,w) \notin R \}
\]
has $m$-measure zero. Regarding measurability, also $A_{-}$ is a $\sigma$-compact set. 
Being the difference of two $\sigma$-compact sets, the set $\mathcal{T}_{e} \setminus \left( A_{+} \cup A_{-} \right)$
is $\sigma$-compact as well.

The next proposition clarifies the importance of absence of branching geodesic in the sense of Proposition \ref{P:nobranch}.

\begin{theorem}\label{T:equivalence}
The set of transport rays $R\subset X \times X$ is an equivalence relation on the set
\[
\mathcal{T}_{e} \setminus \left( A_{+} \cup A_{-} \right).
\]
\end{theorem}

\begin{proof}
First, for all $x \in P_{1}(\Gamma)$, $(x,x) \in R$. If $x,y \in \mathcal{T}_{e}$ with $(x,y) \in R$, then by definition of $R$, it follows straightforwardly that $(y,x) \in R$.

So the only property needing a proof is transitivity. Let $x,z,w \in \mathcal{T}_{e} \setminus \left( A_{+} \cup A_{-} \right)$
be such that $(x,z), (z,w) \in R$ with $x,z$ and $w$ distinct points. The claim is $(x,w) \in R$.
So we have 4 different possibilities: the first one is 
\[
z\in \Gamma(x), \quad w \in \Gamma(z).
\]
This immediately implies $w \in \Gamma(x)$ and therefore $(x,w) \in R$.
The second possibility is 
\[
z\in \Gamma(x), \quad z \in \Gamma(w),
\]
that can be rewritten as $(z,x), (z,w) \in \Gamma^{-1}$. Since $z \notin A_{-}$, necessarily $(x,w) \in R$. 
Third possibility: 
\[
x\in \Gamma(z), \quad w \in \Gamma(z),
\]
and since $z \notin A_{+}$ it follows that $(x,w) \in R$.
The last case is 
\[
x\in \Gamma(z), \quad z \in \Gamma(w),
\]
and therefore $x \in \Gamma(w)$, hence $(x,w) \in R$ and the claim follows.
\end{proof}

\section{Structure of $d$-monotone sets}
Theorem \ref{T:equivalence} says that the right set to look at in order to perform a reduction of the Monge problem 
to a family of $1$-dimensional Monge problem is
\[
\mathcal{T} : = \mathcal{T}_{e} \setminus (A_{+} \cup A_{-}),
\]
and we will refer to $\mathcal{T}$ as the \emph{transport set}.

%In this section we reduce the optimal transportation problem to a family one-dimensional transport problem.
%The one-dimensional domains will be the equivalence classes of $R$.
The next step is to show that each equivalence class of $R$ is formed by a single geodesic.

\begin{lemma}\label{L:singlegeo}
Fix any $x \in \mathcal{T}$. Then for any $z,w \in R(x)$  there exists 
$\gamma \in G \subset \Geo(X)$ such that 
\[
\{ x, z,w \} \subset \{ \gamma_{s} : s\in [0,1] \}.
\]
If $\hat \gamma \in G$ enjoys the same property, then between the two sets
\[
\{ \hat \gamma_{s} : s \in [0,1] \}, \quad \{ \gamma_{s} : s \in [0,1] \},
\]
an inclusion must hold.
\end{lemma}

Since $G = \{ \gamma \in \Geo(X) : (\gamma_{0},\gamma_{1}) \in \Gamma \}$, Lemma \ref{L:singlegeo} states that 
as soon as we fix an element $x$ in $\mathcal{T}_{e} \setminus (A_{+} \cup A_{-})$ and we pick two elements $z,w$ in the same equivalence class of $x$, then these three points are aligned on 
a geodesic $\gamma$ whose image is again all contained in the same equivalence class $R(x)$. 
Moreover if there is another geodesic $\hat \gamma$, different from $\gamma$, containing the three points, then either
\[
\{ \hat \gamma_{s} : s \in [0,1] \} \subset \{ \gamma_{s} : s \in [0,1] \},
\]
or $\{\gamma_{s} : s \in [0,1] \} \subset \{ \hat \gamma_{s} : s \in [0,1] \}$.

\begin{proof}
The proof is quite similar to the proof of Theorem \ref{T:equivalence}.
First assume that $x, z$ are $w$ all distinct points otherwise the claim follows trivially. 
Consider different cases.

First case: $z \in \Gamma(x)$ and $w \in \Gamma^{-1}(x)$. Then by $d$-cyclical monotonicity
\[
d(z,w) \leq d(z,x) + d(x,w) = \f^{d}(w) - \f^{d}(z) \leq d(z,w).
\]
Hence $z,x$ and $w$ lie on a geodesic. 

Second case: $z,w \in \Gamma(x)$. Without loss of generality $\f^{d}(x) \geq \f^{d}(w) \geq \f^{d}(z)$. Since  in the proof of 
Lemma \ref{L:geoingamma} we have already excluded the case $\f^{d}(w) = \f^{d}(z)$, we assume $\f^{d}(x) > \f^{d}(w) > \f^{d}(z)$. 
Then if there would not exists any geodesic $\gamma \in G$ with $\gamma_{0} = x$ and $\gamma_{1} = z$ and $\gamma_{s} = w$,
there will be $\gamma \in G$ with $(\gamma_{0},\gamma_{1}) = (x,z)$ and $s \in (0,1)$ such that 
\[
\f^{d}(\gamma_{s}) = \f^{d}(w), \qquad \gamma_{s} \in \Gamma(x),  \qquad \gamma_{s} \neq w.
\]
As observed in the proof of Lemma \ref{L:geoingamma}, this would imply that  $(\gamma_{s},w) \notin R$ and 
since $x \notin A_{+}$ this would be a contradiction. Hence the second case follows.

The remaining two cases follow with the same reasoning, exchanging the role of $\Gamma(x)$ with the one of $\Gamma^{-1}(x)$.
The second part of the statement follows now easily.
\end{proof}

The next step is to decompose the reference measure $m$ restricted to $\mathcal{T}$ 
with respect to the partition given by $R$, that is 
\[
\{ [x]   \}_{x \in \mathcal{T}} = 
\{ y \in \mathcal{T}: (x,y) \in R \}_{x \in \mathcal{T}}.
\]
In order to use Disintegration Theorem, we need to construct the quotient map 
\[
f : \mathcal{T} \to \{ [x]   \}_{x \in \mathcal{T}}
\] 
associated to the equivalence relation $R$. To give a precise statement we need to introduce some terminology.

A \emph{cross-section of an equivalence relation $E$} is a set $S \subset X$ such that the intersection of $S$ with each equivalence class of $E$ is a singleton. 
A \emph{section of an equivalence relation $E$} is a map $f : X \to X$ such that for any $x,y \in X$ it holds
\[
(x,f(x)) \in E, \qquad (x,y)\in E \Rightarrow f(x) =f(y).
\]
Note that to each section $f$ is canonically associated a cross-section
\[
S = \{ x \in X : x=f(x) \}.
\]

The following result is taken from \cite{biacava:streconv}, first part of Section 4. 
There the result is proved under the additional assumption of non-branching. That assumption is only used to deduce that 
each equivalence class of $R$ is a single geodesic. We have proved this property in Lemma \ref{L:singlegeo}, 
so we don't need it again.
\begin{proposition}
\label{P:sicogrF}
There exists an $m$-measurable cross section 
\[
f : \mathcal{T} \to \mathcal{T}
\]
for the equivalence relation $R$. 
\end{proposition}
Since
\[
S = f(\mathcal{T} ) = \{ x \in \mathcal{T}  : d(x,f(x)) = 0 \},
\]
it follows that $S$ is $m$-measurable. We can also consider the quotient measure in the following way
\[
q : = f_{\sharp}  \, m\llcorner_{\mathcal{T}} .
\]
By inner regularity of compact sets, there exists a $\sigma$-compact set $\mathcal{S} \subset S$ such that 
$q(S \setminus \mathcal{S}) = 0$. Being $\mathcal{S}$ a Borel set, the Disintegration of $m$ restricted to $f^{-1}(\mathcal{S})$
is strongly consistent: 
\[
m\llcorner_{f^{-1}(\mathcal{S})} = \int_{\mathcal{S}} m_{\alpha} q(d\alpha), \qquad m_{\alpha}(f^{-1}(\alpha)) = \| m_{\alpha} \|, \ q-a.e.
\ \alpha \in \mathcal{S}.
\]
Since $q(S \setminus \mathcal{S}) = 0$ reads also as 
\[
m(\mathcal{T} \setminus f^{-1}(\mathcal{S})) = 0,
\]
the previous disintegration formula becomes
\begin{equation}\label{E:disint}
m\llcorner_{\mathcal{T}} = 
\int_{\mathcal{S}} m_{\alpha} q(d\alpha), \qquad m_{\alpha}(f^{-1}(\alpha)) = \| m_{\alpha} \|, \ q-a.e.
\ \alpha \in \mathcal{S}.
\end{equation}

We conclude this section by recalling a definition of \cite{biacava:streconv}, Section 4.
\begin{definition}[Ray map]
\label{D:mongemap}
Define the \emph{ray map}  $g: \mathcal{S} \times \R \to \mathcal{T}$ via the formula
\begin{align*}
\gr (g) : = &~ \Big\{ (x,t,y): x \in \mathcal{S}, t \in [0,+\infty), y \in \Gamma(x) \cap \mathcal{T}_{e}   \cap \{d(x,y) = t\} \Big\} \crcr
&~ \cup \Big\{ (x,t,y): x \in \mathcal{S}, t \in (-\infty,0), y \in \Gamma^{-1}(x) \cap \mathcal{T}_{e}\cap \{d(x,y) = -t\} \Big\} \crcr
=&~ \gr(g^+) \cup \gr(g^-).
\end{align*}
\end{definition}
Hence the ray map associate to each $y \in \mathcal{S}$ and $t$ the unique element in 
$\Gamma(y)\cap   \mathcal{T}_{e}$ at distance $t$ from $y$
if $t$ is positive or the unique element in 
$\Gamma^{-1}(y) \cap  \mathcal{T}_{e}$ 
at distance $-t$, if $t$ is negative. Thanks to Theorem \ref{T:equivalence} and Lemma \ref{L:singlegeo}, 
the ray map $g$ is well defined.

Next we list few regularity properties enjoyed by $g$.
\begin{proposition}
\label{P:gammaclass}
The following holds.
\begin{enumerate}
\item The restriction of $\gr( g)$ to $\mathcal{S} \times \R$ is analytic, and therefore the map is Borel.
\item The range of $g$ is $\mathcal{T} \cup a \cup b$.
\item $t \mapsto g(y,t)$ is a $d$ $1$-Lipschitz $\Gamma$-order preserving for $y \in \mathcal{S}$.
\item $(t,y) \mapsto g(y,t)$ is bijective on $\mathcal{T}$, and its inverse is
\[
x \mapsto g^{-1}(x) = \big( f(x),\pm d(x,f(x)) \big)
\]
where $f$ is the quotient map of Proposition \ref{P:sicogrF} and the positive/negative sign depends on 
$x \in \Gamma(f(x))$ or $x \in \Gamma^{-1}(f(x))$.
\end{enumerate}
\end{proposition}

In this Section we have obtained the first result of this note. In particular we have shown that given a $d$-monotone set 
$\Gamma$, neglecting a set of $m$-measure zero, the set of all those points moved by $\Gamma$, denoted with $\mathcal{T}_{e}$, 
can be written as the union of a family of  disjoint geodesics. 

We include the result in the next theorem.

\begin{theorem} Let $(X,d,m)$ be a metric measure space verifying $\RCD(K,N)$ for some $K,N \in \R$, with $N\geq 1$.
Let moreover $\Gamma$ be a $d$-cyclically monotone set as \eqref{E:Gamma} and let $\mathcal{T}_{e}$ be the set of all points moved by $\Gamma$ as in Definition \ref{D:transport}.
Then there exists $\mathcal{T} \subset \mathcal{T}_{e}$ that we call transport set such that 
\[
m(\mathcal{T}_{e} \setminus \mathcal{T}) = 0, \quad 
\]
and for all $x \in \mathcal{T}$, the transport ray $R(x)$ is formed by a single geodesic and for $x\neq y$, both in $\mathcal{T}$, either $R(x) = R(y)$
or $R(x) \cap R(y)$ is contained in the set of initial points $a \cup b$ as Definition \ref{D:transport}.
\end{theorem}

\section{Regularity of disintegration}\label{S:regularity}

Now we show that for $q$-a.e. $y \in \mathcal{S}$
\begin{equation}\label{E:ac}
m_{y} \ll  \left( g(y,\cdot) \right)_{\sharp} \mathcal{L}^{1}.
\end{equation}

Property \eqref{E:ac} is linked to the 
behavior in time of the measure of evolving subsets of $\mathcal{T}$, where the ``evolving subsets'' has to be made precise.

Since in $\RCD(K,N)$-spaces a concavity estimate for densities of $L^{2}$-geodesics in $\mathcal{P}_{2}(X,d,m)$ holds, 
it is natural to look for a definition of evolution inside the transport set 
where an $L^{2}$-structure can come into play.

\begin{lemma}\label{L:evo1}
For each $C \subset \mathcal{T}$ and $\delta \in \R$ the set 
\[
\left( C \times \{ \f^{d}= \delta \} \right) \cap \Gamma,
\]
is $d^{2}$-cyclically monotone.
\end{lemma}

\begin{proof} The proof follows easily from Lemma \ref{L:12monotone}, indeed the set 
$\left( C \times \{ \f^{d}= c \} \right) \cap \Gamma$ is trivially a subset of $\Gamma$ and whenever  
\[
(x_{0},y_{0}), (x_{1},y_{1}) \in \left( C \times \{ \f^{d}= \delta \} \right) \cap \Gamma,
\]
then $(\f^{d}(y_{1}) -  \f^{d}(y_{0})  ) \cdot (\f^{d}(x_{1}) -  \f^{d}(x_{0})  ) = 0$.
\end{proof}
We can deduce the following

\begin{corollary}\label{C:evo1}
For each $C \subset \mathcal{T}$ and $\delta \in \R$
define 
\[
C_{\delta}: =P_{1}(\left( C \times \{ \f^{d}= \delta \} \right) \cap \Gamma).
\]
Then if $m(C_{\delta}) > 0$, there exists a unique $\nu \in \Opt$ such that 
\begin{equation}\label{E:12mappa}
\left( e_{0} \right)_{\sharp} \nu = m( C_{\delta} )^{-1} m\llcorner_{C_{\delta}},  
\qquad (e_{0},e_{1})_{\sharp}( \nu ) \left(   \left( C \times \{ \f^{d}= \delta \} \right) \cap \Gamma \right) = 1.
\end{equation}
\end{corollary}

\bigskip
\bigskip
From Corollary \ref{C:evo1} and the $\RCD$ condition, we infer the existence of a map $T_{C,\delta}$ depending on $C$ and $\delta$ such that 
\[
\left( Id,T_{C,\delta} \right)_{\sharp} \left( m( C_{\delta} )^{-1} m\llcorner_{C_{\delta}} \right) = (e_{0},e_{1})_{\sharp} \nu.
\]
Taking advantage of the ray map $g$, we define a convex combination between the identity map and 
$T_{C,\delta}$ as follows:
\[
C_{\delta} \ni x \mapsto \left(T_{C,\delta}\right)_{t}(x) = \{ z \in \Gamma(x) : d(x,z) = t \cdot d(x,T_{C,\delta}(x)) \}.
\]
Since $C \subset \mathcal{T}$, the map $\left(T_{C,\delta}\right)_{t}$ is well defined.
We then define the evolution of any subset $A$ of $C_{\delta}$ in the following way:
\[
[0,1] \ni t \mapsto \left(T_{C,\delta}\right)_{t}(A).
\]
In particular from now on we will adopt the following notation:  
\[
A_{t} : = \left(T_{C,\delta} \right)_{t}(A), \qquad \forall A \subset C_{\delta}, \ A \ \textrm{ compact}. 
\]

So for any $C \subset \mathcal{T}$ compact and  $\delta \in \R$ we have defined an evolution 
for compact subsets of $C_{\delta}$. The definition of the evolution depends both on $C$ and $\delta$.

\begin{remark}\label{R:regularity2}
Here we spend few lines on the measurability of the maps involved in the definition of evolution of sets. 
First note that since $\Gamma$ is closed, if $C$ is compact the same holds for $C_{\delta}$.
Moreover 
\[
\gr (T_{C,\delta}) = \left( C \times \{ \f^{d}= \delta \} \right) \cap \Gamma,
\]
hence $T_{C,\delta}$ is continuous. Moreover 
\[
\left(T_{C,\delta}\right)_{t}(A) = P_{2} \left( \{(x,z) \in \Gamma \cap (A \times X) : d(x,z) = t\cdot d(x,T_{C,\delta}(x))  \}\right),
\]
hence if $A$ is compact, the same holds for $\left(T_{C,\delta}\right)_{t}(A)$. It is also possible to show that 
\[
[0,1] \ni t \mapsto m(\left(T_{C,\delta}\right)_{t}(A)) 
\]
is $m$-measurable. We refer to \cite{biacava:streconv}, Lemma 5.2, for its proof.
\end{remark}

The next result gives  quantitative information on the behavior of the map $t \mapsto m(A_{t})$. 
The statement will be given assuming the lower bound on the generalized Ricci curvature $K$ to be positive. 
Analogous estimates holds for any $K \in \R$.

\begin{proposition}\label{P:mcp}
For each $C \subset \mathcal{T}$ and $\delta \in \R$ such that $m(C_{\delta}) >0$, it holds
\begin{equation}\label{E:mcp}
m(  A_{t}  ) \geq  (1-t) \cdot \inf_{x\in C_{\delta}}   \left(\frac{\sin ((1-t) d(x,T_{C,\delta}(x) )\sqrt{K/(N-1)}   ) }
{\sin ( d(x,T_{C,\delta}(x))\sqrt{K/(N-1)}   )} \right)^{N-1}  m(A), 
\end{equation}
for all $t \in [0,1]$ and $A \subset C_{\delta}$ compact set.
\end{proposition}

\begin{proof}
The proof of \eqref{E:mcp} is obtained by the standard method of approximation with Dirac deltas of the second marginal.
More precisely: consider a sequence $\{ y_{i} \}_{i \in \enne} \subset \{\f^{d} = \delta\}$ dense in $T_{C,\delta}(C_{\delta})$.
For each $I \in \N$, define the family of sets 
\[
E_{i,I} : = \{ x \in C_{\delta} : d(x,y_{i}) \leq  d(x,y_{j}) , j =1,\dots, I \},
\]
for $i =1, \dots, I$.  Then for all $I \in \N$, by the same argument of Lemma \ref{L:evo1}, the set 
\[
\bigcup_{i =1}^{I} E_{i,I}\times \{ y_{i} \} \subset X \times X, 
\]
is $d^{2}$-cyclically monotone. Since $\RCD(K,N)$ implies $\mathsf{MCP}(K,N)$, see \cite{giglirajasturm:optimalmaps} 
and \cite{cavasturm:MCP}, and the $L^{2}$ optimal plans are unique, the estimate \eqref{E:mcp} is proved letting $I \to \infty$.
\end{proof}

\subsection{Absolute continuity of conditional measures}

We are now ready to prove that for $q$-a.e. $y \in \mathcal{S}$ \eqref{E:ac} holds. 
For each $y \in \mathcal{S}$ we consider 
the Radon-Nikodym derivative of $m_{y}$ with respect  to $g(y,\cdot)_\sharp \mathcal{L}^1$:
\[
m_y = r(y,\cdot) g(y,\cdot)_\sharp \mathcal{L}^1 + \omega_y, \quad \omega_y \perp g(y,\cdot)_\sharp \mathcal{L}^1
\]

\begin{lemma}\label{L:dec}
There exists a Borel set $C\subset X$ such that
\[
\mathcal{L}^{1} \big( g^{-1} (C) \cap (\{y\} \times \R) ) \big) = 0, \qquad \omega_y = m_y \llcorner_C,
\]
for $q$-a.e. $y \in \mathcal{S}$.
\end{lemma}

\begin{proof}
Consider the measure
\[
\lambda = g_\sharp (q \otimes \mathcal L^1),
\]
and compute the Radon-Nikodym decomposition
\[
m  = \frac{D m}{D \lambda} \lambda + \omega.
\]
Then there exists a Borel set $C$ such that $\omega = m \llcorner_C$ and $\lambda(C)=0$. The set $C$ proves the Lemma. Indeed $C = \cup_{y \in [0,1]} C_{y}$ where $C_{y} = C \cap f^{-1}(y)$ is such that $m_y \llcorner_{C_{y}} = \omega_{y} $ and 
$\left( g(y,\cdot)_{\sharp}\mathcal{L}^{1} \right)(C_{y})=0$ for $q$-a.e. $y \in \mathcal{S}$.
\end{proof}

\begin{theorem}
\label{T:a.c.}
For $q$-a.e. $y \in \mathcal{S}$, the conditional probabilities $m_y$ are absolutely continuous w.r.t. $g(y,\cdot)_\sharp \mathcal{L}^1$.
\end{theorem}

\begin{proof}
{\it Step 1.}
Take as $C$ the set constructed in Lemma \ref{L:dec} and suppose by contradiction that
\[
m(C) > 0,
\]
and we already know that $q \otimes \mathcal{L}^1 (g^{-1}(C)) = 0$.

We want to find $\hat C \subset C$ compact set with $m(\hat C) > 0$ and $\delta \in \R$ such that for each $z \in \hat C$ there exists 
$w \in \Gamma(z)$ such that $\f^{d}(w) = \delta$.
Possibly localizing, we can assume that for each $z \in C$
\[
\sup\{ d(x,y) : x,y \in R(z) \} \geq M, 
\]
and $C \subset B_{\ve}(\bar z)$ for some $\bar z \in \mathcal{T}$ and $\ve \leq M$.
Then the $1$-Lipschitz property of $\f^{d}$ implies the existence of such $\delta$ and $\hat C$ compact, such that $m(\hat C)>0$ 
with 
\[
\hat C = \hat C_{\delta}.
\]
We can therefore consider the evolution in time of $\hat C$ with respect to $\delta$, $(\hat C)_{t}$ and  for $t \in [0,1]$
the inequality \eqref{E:mcp} holds.

{\it Step 2.}
Since $\hat C \subset C$, it still holds that
\[
q \otimes \mathcal{L}^{1} (g^{-1}(\hat C)) = 0.
\]
In particular, for all $t \in [0,1]$ it follows that
\[
q \otimes \mathcal{L}^1 (g^{-1}((\hat C)_t)) \leq q \otimes \mathcal{L}^1 (g^{-1}(\hat C)) = 0.
\]
Indeed since the evolution of $\hat C$ runs along the transport rays, $f(\hat C_{t}) = f(\hat C)$ where $f$ is the quotient map.
Moreover on each single transport ray, the evolution of $\hat C$ is just the linear contraction to a single point. Hence the inequality follows.
\\
We also need the following object: for each $y \in f(\hat C) = P_{1}(g^{-1}(\hat C))$ there exists only one $\tau \in \R$, say $\tau(y)$ such that 
\[
g(y,\tau) \in  \{ \f^{d} = \delta\}.
\]
To underline the ($m$-measurable) dependence of $\tau$ on $y$ and $\delta$, we will denote it with $\tau(y,\delta)$. 
With this notation, we can express $\hat C_{t}$ in the following way:
\[
\hat C_{t} = g \left(   \left\{ (y, \tau +(\tau(y,\delta) - \tau) t ) : (y,\tau) \in g^{-1}(\hat C) \right\}  \right), 
\]
and consequently 
\[
g^{-1}\left(\hat C_{t} \right)=    \left\{ (y, \tau +(\tau(y,\delta) - \tau) t ) : (y,\tau) \in g^{-1}(\hat C) \right\}. 
\]
Then by Fubini-Tonelli Theorem and Proposition \ref{P:mcp}
\begin{align}\label{A:chain}
0< &~ \int_{(0,1/2)} m(\hat C_t) dt  =  \int_{(0,1/2)} \bigg( \int_{g^{-1}(\hat C_t)} \left( (g^{-1})_\sharp m \right) (dyd\tau) \bigg) dt \crcr
=&~ \big( (g^{-1})_\sharp  m   \big) \otimes\mathcal{L}^1  \Big( \Big\{ (y,\tau,t) \in \mathcal{S}\times \R \times [0,1/2]: (y,\tau) \in g^{-1}(\hat C_{t}) \Big\} \Big) \crcr
= &~ \int_{\mathcal{S} \times \R} \mathcal{L}^1 \left( \left\{ t \in [0,1/2] :   (y, \tau ) \in g^{-1}(\hat C_{t})  \right\} \right) 
\left( g^{-1}_{\sharp} m \right)(dyd\tau), \crcr
= &~ \int_{\mathcal{S} \times \R} \mathcal{L}^1 \left( \left\{ t \in [0,1/2] :   \left(y, \frac{\tau - \tau(y,\delta)t }{1-t} \right) \in
g^{-1}(\hat C)  \right\} \right) \left( g^{-1}_{\sharp} m \right)(dyd\tau).
%=&~ \int_{S \times \R} \mathcal{L}^1 \big( g^{-1}(C \cap f^{-1}(y)) \big) \left( (g^{-1})_{\sharp} m \right) (dyd\tau) \crcr
%=&~ \int_{S} \mathcal{L}^1 \big( g^{-1}(C \cap f^{-1}(y)) \big) m(dy) = 0.
\end{align}
Now by definition of $C$, for $q$-a.e. $y \in \mathcal{S}$,
\[
\mathcal{L}^{1}\left( \{ t\in \R : (y,t) \in g^{-1}(\hat C) \} \right)= 0.
\]
Since the function $[0,1/2] \ni t \mapsto \left( \tau- \tau(y,\delta)t \right)/(1-t)$
is smooth, also the following holds
\[
\mathcal{L}^1 \left( \left\{ t \in [0,1/2] :   \left(y, \frac{\tau - \tau(y,\delta)t }{1-t} \right) \in g^{-1}(\hat C)  \right\} \right)  = 0,
\]
for $q$-a.e.  $y \in \mathcal{S}$.

Since $\left(P_{1} \right)_{\sharp} \left( (g^{-1})_{\sharp} m \right)  = q$ it follows that the last integral in \eqref{A:chain} is null,
giving a contradiction with the strictly positive sign of the first one.
\end{proof}

\section{Existence of solution to the Monge problem}

Using Theorem \ref{T:a.c.} we prove the existence of an $m$-measurable map $\hat T : X \to X$ such that 
\[
\int_{X} d(x,\hat T(x)) \mu_{0}(dx) =  
\inf_{T_{\sharp}\mu_{0} = \mu_{1}}  \int_{X} d(x,T(x)) \mu_{0}(dx),
\]
with $\hat T_{\sharp} \mu_{0} = \mu_{1}$, provided $\mu_{0}$ is absolutely continuous with respect to $m$. 
So assume $\mu_{0} = \r_{0} m$. 
\\
\\

Justified by Lemma \ref{L:mapoutside}, extension \eqref{E:extere} and Proposition \ref{P:nobranch}, we assume that $\mu_{0}(\mathcal{T})=\mu_{1}(\mathcal{T}_{e}) =1$. 
Then \eqref{E:disint} gives that 
\begin{equation}\label{E:disintmu}
\mu_{0} = \r_{0} m\llcorner_{\mathcal{T}} =\int_{\mathcal{S}} \r_{0}m_{\alpha} q(d\alpha) = \int_{\mathcal{S}} \mu_{0,y} q_{\mu_{0}}(dy),
\end{equation}
where $\mu_{0,y} = c(y) \r_{0}m_{\alpha}$ with $c(y)$ normalizing constant, and $q_{\mu_{0}} = c(y)^{-1}q$. 

Since $R$ is an equivalence relation only on $\mathcal{T}$ and a priori $\mu_{1}(\mathcal{T}_{e}\setminus \mathcal{T} ) > 0$
is not excluded, \eqref{E:disintmu} is not automatically true for $\mu_{1}$.
To get a disintegration for $\mu_{1}$ we pass through a disintegration of a given $\eta \in \Pi_{opt}(\mu_{0},\mu_{1})$.

\begin{lemma}\label{L:disinteta}
Let $\eta \in \Pi_{opt}(\mu_{0},\mu_{1})$ be given, then the following disintegration formula holds:
\[
\eta = \int_{\mathcal{S}} \eta_{y} q_{\mu_{0}}(dy), \qquad \eta_{y} \in \mathcal{P}( \left(R(y) \cap \mathcal{T}\right) \times R(y)), 
\]
with $P_{1\,\sharp} \eta_{y} = \mu_{0,y}$.
\end{lemma}

\begin{proof}
Since $\eta(\mathcal{T} \times X \cap \Gamma ) = 1$, we want to find the right partition of 
$\left( \mathcal{T} \times X \right) \cap \Gamma$ and that can be done via the partition $\{R(y)\}_{y \in \mathcal{S}}$ of $\mathcal{T}$:
\[
\left( \mathcal{T} \times X \right) \cap \Gamma  = \bigcup_{y \in \mathcal{S}} \left( R(y) \cap \mathcal{T} \right) \times X.
\]
Then by Disintegration Theorem it follows that 
\[
\eta = \int_{\mathcal{S}} \eta_{y} q_{\eta}(dy), \qquad \eta_{y} \in \mathcal{P}(\left( \left(R(y)\cap \mathcal{T}\right) \times X \right) \cap \Gamma). 
\]
Since $\left(\left( R(y)\cap \mathcal{T}\right) \times X \right) \cap \Gamma) \subset \left(R(y)\cap \mathcal{T}\right) \times R(y)$,
to prove the claim we need to show that:
\[
q_{\eta} = q_{\mu_{0}}.
\]
Since for $I \subset \mathcal{S}$ it holds that 
\[
\mu_{0} \left( f^{-1}(I) \right) = \eta \left( f^{-1}(I) \times X \right) 
=  \eta \left( \left\{ (x,y) \in \Gamma : x \in \mathcal{T}, f(x) \in I   \right\} \right),
\]
the claim follows.
\end{proof}

We can obtain a dimensional reduction also for $\mu_{1}$:
\begin{equation}
\label{E:mu1}
\mu_{1} = P_{2\,\sharp} \eta = \int_{\mathcal{S}} (P_{2})_{\sharp }\eta_y \,q_{\mu_{0}}(dy) = 
\int_{\mathcal{S}} \mu_{1, y} \,q_{\mu_{0}}(dy).
\end{equation}
In particular $\eta_y \in \Pi(\mu_{0,y},\mu_{1,y})$ is $d$-cyclically monotone (and hence optimal, because $R(y)$ is one dimensional) for $q_{\mu_{0}}$-a.e. $y$. If $\mu_{1}(\mathcal{T}) = 1$, 
then \eqref{E:mu1} is the disintegration of $\mu_{1}$ w.r.t. $R$.

\begin{remark}
Since $\mu_{0} (\mathcal{T}) =1$ and for each $y \in \mathcal{T}$ the set $R(y)$ is 1-dimensional and 
for $y \in \mathcal{S}$ we have proved that $R(y) = g(y,\R)$, 
without loss of generality we can refer to $\mu_{0,y}$ and $\mu_{1,y}$ as Borel probability measures over $\R$.
It will be clear from the context whether we still refer to them as measures over $X$.
\end{remark}

We now prove the existence of a solution to Monge minimization problem.

\begin{theorem}
\label{T:mongeff}
Let $(X,d,m)$ be a metric measure space verifying $\RCD(K,N)$ for $N<\infty$. Let $\mu_{0},\mu_{1} \in \mathcal{P}(X)$ 
with $W_{1}(\mu_{0},\mu_{1}) < \infty$ and $\mu_{0}\ll m$. 
Then there exists a Borel map $T: X \to X$ such that $T_{\sharp} \mu_{0} = \mu_{1}$ and 
\[
\int_{X} d(x,T(x)) \mu_{0}(dx) = \int_{X \times X} d(x,y) \eta(dxdy),
\]
for any $\eta \in \Pi_{opt}(\mu_{0},\mu_{1})$.
\end{theorem}

\begin{proof}
{\it Step 1.} 
By means of the map $g^{-1}$, we reduce to a transport problem on $\mathcal{S} \times \R$, with cost
\[
c((y,s),(y',t)) =
\begin{cases}
|t - s| & y = y' \crcr
+ \infty & y \not= y'
\end{cases}
\]
It is enough to prove the theorem in this setting under the following assumptions: $\mathcal{S}$ compact 
and $\mathcal{S} \ni y \mapsto (\mu_{0,y},\mu_{1,y})$ weakly continuous. 

From the weak continuity of the map $y \mapsto (\mu_{0,y},\mu_{1,y})$, it follows that the maps
\[
(y,t) \mapsto H(y,t) := \mu_{0,y}((-\infty,t)), \ \ (y,t) \mapsto F(y,t) := \mu_{1,y}((-\infty,t))
\]
are lower semi-continuous.  Both are increasing in $t$ and $H$ is continuous in $t$.
The map $T$ defined as in Theorem \ref{T:oneDmonge} by
\[
T(y,s) := \Big( y, \sup \big\{ t : F(y,t) \leq H(y,s) \big\} \Big)
\]
is Borel. In fact, for $A$ Borel,
\[
T^{-1}(A \times [t,+\infty)) = \big\{ (y,s) : y \in A, H(y,s) \geq F(y,t) \big\} \in \mathcal{B}(S \times \R).
\]

{\it Step 2.} 
Since $\mu_{0,y}$ has no atoms for $q_{\mu_{0}}$-a.e. $y \in \mathcal{T}$,
$T(y,\cdot)$ is optimal for the transport problem between $\mu_{0,y}$ and $\mu_{1,y}$ with cost $|\cdot|$.
By $d$-cyclical monotonicity, the same holds for $\eta_{y}$. Then using Lemma \ref{L:disinteta}, it follows that 
\begin{align*}
\int d(x,T(x)) \mu_{0}(dx) = &~ \int_{\mathcal{S}} \int_{\R} |t-T(y,t)| \mu_{0,y}(dt) q_{\mu_{0}}(dy) \crcr
= &~ \int_{\mathcal{S}} \int_{X \times X} d(x,z) \eta_{y}(dxdz) q_{\mu_{0}}(dy)  \crcr
=&~ \int d(x,z) \eta(dxdz).
\end{align*}
The claim follows.
\end{proof}

Since a priori 
\[
\inf_{\eta \in \Pi_{\mu_{0},\mu_{1}}} \int d(x,z) \eta(dxdz)  \leq \inf_{T_{\sharp}\mu_{0} = \mu_{1}} \int d(x,T(x)) \mu_{0}(dx),  
\]
we have also proved \eqref{E:K=M}.

\begin{corollary}\label{C:M=K}
Let $(X,d,m)$ be a metric measure space verifying $\RCD(K,N)$ for $N<\infty$. Let $\mu_{0},\mu_{1} \in \mathcal{P}(X)$ 
with $W_{1}(\mu_{0},\mu_{1}) < \infty$ and $\mu_{0}\ll m$. 
Then
\[
\min_{\eta \in \Pi_{\mu_{0},\mu_{1}}} \int d(x,z) \eta(dxdz)  = \min_{T_{\sharp}\mu_{0} = \mu_{1}} \int d(x,T(x)) \mu_{0}(dx),  
\]
\end{corollary}

\section*{Appendix: Estimate on the one dimensional density}

Here we include a result that is not strictly necessary in the proof of Theorem \ref{T:mongeff} and Corollary \ref{C:M=K}, 
but it will be used in a future publication to extend the results of \cite{cav:decomposition} 
to the case of $\RCD(K,N)$ and therefore to remove the non-branching assumption.

Using the curvature property $\RCD(K,N)$ verified by $(X,d,m)$ and the estimate \eqref{E:mcp}, we can prove regularity property for the density of 
$g(y,\cdot)^{-1}_{\sharp} m_{y}$ with respect to $\mathcal{L}^{1}$.
So we introduce the function $h : \mathcal{S} \times \R \to [0,\infty)$ so that
\[
m= g_{\sharp} \left( h\, q \otimes \mathcal{L}^{1} \right).
\]
We will prove some estimate for the map $t \mapsto h(y,t)$ for $q$-a.e. $y \in \mathcal{S}$. Again the estimates proved 
here are obtained assuming $K > 0$, anyway analogous calculations hold for any $K \in \R$ after suitable modifications.

Each ray $R(y)$ for $q$-a.e. $y \in \mathcal{S}$ is invariant 
for the evolution for compact subsets of the transport set $\mathcal{T}$, introduced in Section \ref{S:regularity}.  
Then, using standard arguments, estimate \eqref{E:mcp} can be localized at the level of the density $h$: 
for each compact set $A \subset \mathcal{T}$ 
\begin{align*}
\int_{P_{2}(g^{-1}(A_{t}) )} & h(y,s) \mathcal{L}^{1}(ds)  \crcr
 \geq (1-t) & \left( \inf_{\tau \in P_{2}(g^{-1}(A))}
\frac{\sin( (1-t) |\tau - \sigma|  \sqrt{K/(N-1)} )   }{\sin( |\tau - \sigma|  \sqrt{K/(N-1)} )}  \right)^{N-1}
\int_{P_{2}(g^{-1}(A))} h(y,s) \mathcal{L}^{1}(ds),
\end{align*}
for $q$-a.e. $y \in \mathcal{S}$ such that $g(y,\sigma) \in \mathcal{T}$.

Then using change of variable, one can obtain that for $q$-a.e. $y \in \mathcal{S}$:
\[
h(y,s+|s-\sigma| t ) \geq \left(
\frac{\sin( (1-t) |s - \sigma|  \sqrt{K/(N-1)} )   }{\sin( |s - \sigma|  \sqrt{K/(N-1)} )}  \right)^{N-1}
 h(y,s),
\]
for $\mathcal{L}^{1}$-a.e.  $s \in P_{2}(g^{-1}(R(y)))$ and $\sigma \in \R$ such that $s + |\sigma -s| \in P_{2}(g^{-1}(R(y)))$.
Then it can be rewritten in the following way: 
\[
h(y, \tau ) \geq \left(
\frac{\sin(  ( \sigma - \tau )  \sqrt{K/(N-1)} )   }{\sin( ( \sigma - s )  \sqrt{K/(N-1)} )}  \right)^{N-1} h(y,s),
\]
for $\mathcal{L}^{1}$-a.e. $s \leq \tau \leq  \sigma$ such that $g(y,s), g(y,\tau), g(y,\sigma) \in \mathcal{T}$. 

Since evolution can be also defined backwardly, we have proved the next

\begin{proposition}\label{P:densityestimates}
For $q$-a.e. $y \in \mathcal{S}$ it holds: 
\[
\left( \frac{\sin(  ( \sigma_{+} - \tau )  \sqrt{K/(N-1)} )   }{\sin( ( \sigma_{+} - s )  \sqrt{K/(N-1)} )}  \right)^{N-1} 
\leq \frac{h(y, \tau )} {h(y,s)} 
\leq  \left( \frac{\sin(  (  \tau - \sigma_{-} )  \sqrt{K/(N-1)} )   }{\sin( (s - \sigma_{-}  )  \sqrt{K/(N-1)} )}  \right)^{N-1},
\]
for $\sigma_{-} < s \leq \tau < \sigma _{+}$ such that their image via $g(y,\cdot)$ is contained in $R(y)$.
\end{proposition}


\begin{thebibliography}{10}

\bibitem{ambrgisav:rcd}
L.~Ambrosio, N.~Gigli, and G.~Savar\'e.
\newblock Metric measure spaces with {R}iemannian {R}icci curvature buonded
  from below.
\newblock Preprint, arXiv:1109.0222.

\bibitem{ambprat:crist}
L.~Ambrosio, B.~Kirchheim, and A.~Pratelli.
\newblock Existence of optimal transport maps for crystalline norms.
\newblock {\em Duke Math. J.}, 125(2):207--241, 2004.

\bibitem{ambmondsav:RCD}
L.~Ambrosio, A.~Mondino, and G.~Savar\'e.
\newblock Nonlinear diffusion equations and curvature conditions in metric
  measure spaces.
\newblock in progress, 2013.

\bibitem{sturm:loc}
K.~Bacher and K.T. Sturm.
\newblock Localization and tensorization properties of the
  {C}urvature-{D}imension condition for metric measure spaces.
\newblock {\em J. Funct. Anal.}, 259(1):28--56, 2010.

\bibitem{biacava:streconv}
S.~Bianchini and F.~Cavalletti.
\newblock The {M}onge problem for distance cost in geodesic spaces.
\newblock {\em Comm. Math. Phys}, 318:615 -- 673, 2013.

\bibitem{caffafeldmc}
L.~Caffarelli, M.~Feldman, and R.J. McCann.
\newblock Constructing optimal maps for {M}onge's transport problem as a limit
  of strictly convex costs.
\newblock {\em J. Amer. Math. Soc.}, 15:1--26, 2002.

\bibitem{caravenna:Monge}
L.~Caravenna.
\newblock A proof of {M}onge problem in $\mathbb{R}^n$ by stability.
\newblock {\em Rend. Istit. Mat. Univ. Trieste}, 43:31--51, 2011.

\bibitem{cava:wiener}
F.~Cavalletti.
\newblock The {M}onge problem in {W}iener space.
\newblock {\em Calc. Var. Partial Differential Equations}, 45(1-2):101Ð124,
  2012.

\bibitem{cava:nonconv}
F.~Cavalletti.
\newblock Optimal transport with branching distance costs and the obstacle
  problem.
\newblock {\em SIAM J. Math. Anal.}, 44(1):454Ð482, 2012.

\bibitem{cav:decomposition}
F.~Cavalletti.
\newblock Decomposition of geodesics in the {W}asserstein space and the
  globalization property.
\newblock {\em accepted on GAFA}, 2013.
\newblock preprint arXiv:1209.5909.

\bibitem{cavasturm:MCP}
F.~Cavalletti and K.-T. Sturm.
\newblock Local curvature-dimension condition implies measure-contraction
  property.
\newblock {\em J. Funct. Anal.}, 262:5110 -- 5127, 2012.

\bibitem{champdepasc:Monge}
T.~Champion and L.~De Pascale.
\newblock The {M}onge problem in $\mathbb{R}^d$.
\newblock {\em Duke Math. J.}, 157(3):551--572, 2011.

\bibitem{Erbarkuwstu:RCD}
M.~Erbar, K.~Kuwada, and K.-T. Sturm.
\newblock On the equivalence of the entropic curvature-dimension condition and
  {B}ochner's inequality on metric measure spaces.
\newblock preprint arXiv:1303.4382, 2013.

\bibitem{evagangbo}
L.C. Evans and W.~Gangbo.
\newblock Differential equations methods for the {M}onge-{K}antorovich mass
  transfer problem.
\newblock {\em Current Developments in Mathematics}, pages 65--126, 1997.

\bibitem{feldcann:mani}
M.~Feldman and R.~McCann.
\newblock Monge's transport problem on a {R}iemannian manifold.
\newblock {\em Trans. Amer. Math. Soc.}, 354:1667--1697, 2002.

\bibitem{gigli:laplacian}
N.~Gigli.
\newblock On the differential structure of metric measure spaces and
  applications.
\newblock preprint, arXiv:1205.6622, 2012.

\bibitem{gigli:spliting}
N.~Gigli.
\newblock The splitting theorem in non-smooth context.
\newblock preprint, arXiv:1302.5555, 2013.

\bibitem{giglirajasturm:optimalmaps}
N.~Gigli, T.~Rajala, and K.-T. Sturm.
\newblock Optimal maps and exponentiation on finite dimensional spaces with
  {R}icci curvature bounded from below.
\newblock preprint, arXiv:1305.4849, 2013.

\bibitem{larm}
D.G. Larman.
\newblock A compact set of disjoint line segments in $\erre^{3}$ whose end set
  has positive measure.
\newblock {\em Mathematika}, 18:112--125, 1971.

\bibitem{villott:curv}
J.~Lott and C.~Villani.
\newblock Ricci curvature for metric-measure spaces via optimal transport.
\newblock {\em Ann. of Math.}, 169(3):903--991, 2009.

\bibitem{rajasturm:branch}
T.~Rajala and K.-T. Sturm.
\newblock Non-branching geodesics and optimal maps in strong
  $\mathsf{CD}(k,\infty)$-spaces.
\newblock {\em Calc. Var. Partial Differential Equations}, online first DOI
  10.1007/s00526-013-0657-x, 2013.

\bibitem{Sri:courseborel}
A.~M. Srivastava.
\newblock {\em A course on Borel sets}.
\newblock Springer, 1998.

\bibitem{sturm:MGH1}
K.T. Sturm.
\newblock On the geometry of metric measure spaces.{I}.
\newblock {\em Acta Math.}, 196(1):65--131, 2006.

\bibitem{sturm:MGH2}
K.T. Sturm.
\newblock On the geometry of metric measure spaces.{II}.
\newblock {\em Acta Math.}, 196(1):133--177, 2006.

\bibitem{sudak}
V.N. Sudakov.
\newblock Geometric problems in the theory of dimensional distributions.
\newblock {\em Proc. Steklov Inst. Math.}, 141:1--178, 1979.

\bibitem{trudiwang}
N.~Trudinger and X.J. Wang.
\newblock On the {M}onge mass transfer problem.
\newblock {\em Calc. Var. PDE}, 13:19--31, 2001.

\bibitem{villa:Oldnew}
C.~Villani.
\newblock {\em Optimal transport, old and new}.
\newblock Springer, 2008.

\end{thebibliography}
\end{document}